\documentclass[11pt,a4paper]{article}
\usepackage{preamble}
\usepackage{a4wide}

\title{Disjoint direct product decomposition of permutation groups}
\author{Mun See Chang \and Christopher Jefferson}

\begin{document}

\maketitle

\begin{abstract}
Let $H \leq S_n$ be an intransitive group with orbits $\Omega_1, \Omega_2, \ldots ,\Omega_k$. 
Then certainly $H$ is a subdirect product of the direct product of its projections on each orbit, $H|_{\Omega_1} \times H|_{\Omega_2} \times \ldots \times H|_{\Omega_k}$. Here we provide a polynomial time algorithm for computing the finest partition $P$ of the $H$-orbits such that $H = \prod_{c \in P} H|_c$ and we demonstrate its usefulness in some applications.
\end{abstract}


\section{Introduction}
A \emph{direct product decomposition} of a given group $H$ is an expression of $H$ as a direct product of groups. 
The direct product decomposition is useful for understanding the structure of the group, and to solve problems more efficiently. 
Hence, it is important to find an efficient algorithm for computing such decompositions. 
Kayal and Nezhmetdinovm, in \cite{DPDbyMultTable}, provided a polynomial time algorithm for computing a direct product decomposition of a group $H$ given by its multiplication table, which has input size $|H|^2$. 
We consider computing direct product decompositions of permutation groups given by generating sets, which are usually much smaller than the order of the group. 
Wilson, in \cite{Wilson2010FindingDP} gives a polynomial time solution to such a problem.  
However, as far as we know, this algorithm has not yet been implemented.  

In this paper, we consider a particular type of direct product decomposition for finite permutation groups, which we call a \emph{disjoint direct product decomposition}, which are direct product decompositions of permutation groups where the factors move disjoint sets of points, and so have disjoint supports (See \Cref{defn: internal DDP}). 
We give a polynomial time algorithm for finding disjoint direct product decompositions of a permutation group given by a generating set and also demonstrate the practical efficiency of our algorithm. 
In this paper, we regard the direct product of permutation groups with disjoint supports as a subgroup of the symmetric group on the union of the supports of the direct factors.

\begin{definition}[Disjoint direct product decomposition] \label{defn: internal DDP}\label{defn: ddpd}
Let $H \leq S_n$. 
We say that $H = H_1 \times H_2 \times \ldots \times H_r$ is a \emph{disjoint direct product decomposition} of $H$ if 
it is a direct product decomposition of $H$ and the groups $H_i$ have pairwise disjoint supports. 
Each factor $H_i$ is called a \emph{disjoint direct factor} of $H$, which can be identified as a subgroup of $H$ that fixes all points outside $\Supp(H_i)$.  \\
If there exists a disjoint direct product decomposition $H = H_1 \times H_2 \times \ldots \times H_r$ of $H$ with $r>1$, then we say that $H$ is \emph{d.d.p.\ decomposable}, otherwise we say that $H$ is \emph{d.d.p.\ indecomposable}.   \\
A disjoint direct product decomposition is \emph{finest} if each factor is d.d.p.\ indecomposable.
\end{definition}

Since the disjoint direct product decomposition is more restrictive than a more general decomposition, it can be computed much faster and has many useful applications. 
As we will demonstrate in \Cref{section: experiments}, the disjoint direct product decomposition of a permutation group can be used to greatly speed up various other calculations with permutation groups, such as computing the derived subgroup, number of conjugacy classes and a composition series of a given permutation group. 
Furthermore, calculations that previously could not be completed in a reasonable time frame can be solved very quickly using the disjoint direct product decomposition to subdivide the computation into smaller pieces.

Another important application of disjoint direct product decompositions lies in other areas of computer science, where groups arise from symmetries of combinatorial objects. To reduce the computation time, groups are used to eliminate the symmetries of the objects through a process called \emph{symmetry breaking} \cite{cpbooksymchapter}.

Donaldson et al.\ \cite{donaldson_miller_2006} use the disjoint direct product decomposition to improve the performance of detecting symmetric states in model checking and Grayland et.\ al.\ \cite[Theorem 10]{Graylandpaper} uses the disjoint direct product decomposition when generating symmetry breaking constraints for symmetric problems. 
Grayland et.\ al.\ give an algorithm for symmetry breaking which uses the disjoint direct product of two symmetry groups \cite{Graylandpaper} but otherwise does not consider general direct product decompositions. 
In both of these applications, disjoint direct product decompositions lead to significant speed-ups.

For these applications, the time saved depends on the number of factors in the decomposition. 
Hence we are interested in an algorithm that always computes a \emph{finest} disjoint direct product decomposition. 
By the Krull–Schmidt theorem, any finite group has a unique finest direct product decomposition \cite[Theorem~3.8]{HungerfordAlgebra}.
In \Cref{finest DDPD is unique}, we show that the finest disjoint direct product decomposition of a given finite permutation group is unique. 

The main result of this paper is to provide an efficient algorithm to compute the finest disjoint direct product decomposition of a given permutation group, and hence to prove the following: 

\begin{theorem} \label{main theorem}
Let $H \leq S_n$ be given by a generating set $X$. Then the finest disjoint direct product decomposition of $H$ can be computed in time polynomial in $n \cdot |X|$.
\end{theorem}

Our algorithm behind \Cref{main theorem} manipulates a strong generating set and therefore is fast in practice once a base and strong generating set have been found. Finding a base and strong generating set is an initial part of most permutation group algorithms. Hence, finding a disjoint direct product decomposition will not add significantly to the runtime of these algorithms.

The structure of this paper is as follows. 
In \Cref{section: lit review}, we present some related work in the literature, and the definitions, notation and background knowledge we use later on. 
In \Cref{theory of ddfd}, we present the theoretical framework which we use for the algorithms we present in \Cref{section: algorithm}. 
Also in \Cref{section: algorithm}, we prove \Cref{main theorem}.
Lastly, in \Cref{section: experiments}, we demonstrate how the algorithm can be used to speed up computation in various permutation group theoretic functions in GAP.   

\section{Background and preliminaries}
\label{section: lit review}

If $G$ is a direct factor of $H$, then $G \trianglelefteq H$. So a naive approach to finding its disjoint direct product decomposition is to consider all normal subgroups $N$ of $H$, check if there exists $K$ such that $N \times K=H$, then recursively try to decompose $N$ and $K$. 
While it is possible to optimise this approach, it has worst-case exponential complexity, 
since it requires considering all normal subgroups of $H$, the number of which grows exponentially with $|\Supp(H)|$.

Wilson's polynomial time algorithm in \cite{Wilson2010FindingDP} computes the finest (not necessarily disjoint) direct product decomposition of a given permutation group $H$. 
As far as we are aware, the algorithm has yet to be implemented. 
We show that it is substantially easier to compute the finest disjoint direct product decomposition than the finest direct product decomposition.

Donaldson and Miller in \cite[Section~5.1]{donaldson_miller_2009} present a polynomial algorithm for computing a disjoint direct factor decomposition by considering the given generators. 
They used the observation that, if $H = \langle X \rangle$ and there exists $S \subsetneqq X$ such that 
the support of $S$ and the support of $X \backslash S$ are disjoint, then 
$H = \langle S \rangle \times \langle X \backslash S \rangle$ is a disjoint direct product decomposition.
The method by Donaldson and Miller is a subprocedure of the algorithm we present in \Cref{section: algorithm}. 
However, note that the method in \cite{donaldson_miller_2009} does not guarantee that the decomposition is the finest possible as different choices of $X$ may produce different decompositions.
Donaldson and Miller reported that, using the generators computed from the graph automorphism program they used, this method seems to almost always produce the finest decomposition. 
We hypothesise that these programs almost always produce separable strong generating sets, which we shall define in \Cref{defn: separable}.

In \cite{donaldson_miller_2009}, Donaldson and Miller also present an exponential-time algorithm to compute the finest disjoint direct product decomposition of $H$. 
The algorithm involves recursively computing disjoint direct product decompositions with two factors. 
To construct such a decomposition of $H$, they consider all partitions of the set of $H$-orbits with two cells. 
They test if each partition gives rise to a disjoint direct product decomposition by checking if
$H$ is the direct product of its restriction to the union of orbits in the first cell and its restriction to the union of orbits in the second cell. 
They also made a significant improvement to their algorithm by first considering the restrictions onto pairs of orbits and deciding if they are d.d.p.\ decomposable. 




\ifthenelse{\equal{\thesis}{1}}{}{
\subsection{Notation}}

Throughout the \ifthenelse{\equal{\thesis}{1}}{chapter}{paper} let $\Omega$ be a finite set. 
Let $\overline{i}$ denote the set $\{1,2, \ldots, i\}$. 
\ifthenelse{\equal{\thesis}{1}}{}{
Let $G \leq \Sym(\Omega)$. For a subset $\Delta \subseteq \Omega$, we denote by $G_{(\Delta)}$ the \emph{pointwise stabiliser} of $\Delta$ in $G$. We denote by $\Supp(G)$ the \emph{support} of $G$.} 

\ifthenelse{\equal{\thesis}{1}}{}{
\begin{definition}
Let $G_1, G_2, \ldots, G_k$ be groups and let $G = G_1 \times G_2 \times \ldots \times G_k$. 
A \emph{subdirect product} of $G$ is a subgroup $K \leq G$ such that each projection $\rho_i: K \rightarrow G_i$ onto the $i$-th factor is surjective. That is, $\rho_i(K) = G_i$ for all $1 \leq i \leq k$.
\end{definition}
}

We would like to draw the reader's attention to an elementary result, which we shall repeatedly use later. 

\begin{lemma} \label{DP iff stab gives full group}
Let $H$ be a subdirect product of $G_1 \times G_2$. Then $H = G_1 \times G_2$ if and only if $1 \times G_2 \leq H$. 
\end{lemma}

\begin{proof}
The forward implication is clear. For the backward implication, for all $(g_1, g_2) \in H$, since $(1, g_2) \in H$, we have $(g_1, 1) \in H$. 
The projection of $H$ onto $G_1$ is the whole $G_1$, so $G_1 \times 1 \leq H$. Now since $G_1 \times 1$ and $1 \times G_2$ generate $G_1 \times G_2$ and are both contained in $H$, we have $H = G_1 \times G_2$.
\end{proof}

\ifthenelse{\equal{\thesis}{1}}{}{
\label{section: goursat's}

\Cref{ori goursat} is commonly known as Goursat's lemma. It describes the subgroups of a direct product of groups and appears in the literature in various places, including, for example, \cite{schmidt_subgpLattices, cartesian_decomp}. 
Here we are only concerned about subgroups of direct products that are also subdirect products. 
However, note that every subgroup of a direct product is a subdirect product of the direct product of the images under the projections. 

\begin{theorem}[{\cite{oriGoursat}}] \label{ori goursat}
Let $G_1, G_2$ be groups. Let $H$ be a subdirect product of $G_1 \times G_2$.
Let $\proj_1: H\rightarrow G_1$ and $\proj_2: H\rightarrow G_2$ be the projection maps of $H$ onto $G_1$ and $G_2$ respectively. 
The following hold. 
\begin{enumerate}
    \item Let $N_1 \coloneqq \proj_1( \mathrm{Ker}(\proj_2)) $ and $N_2 \coloneqq \proj_2( \mathrm{Ker}(\proj_1))$. Then $N_1 \trianglelefteq G_1$ and $N_2 \trianglelefteq G_2$. 
    \item \label{H as graph of isom} 
       The map $\theta : G_1/N_1 \rightarrow G_2/N_2$ given by $N_1h_1 \mapsto N_2h_2$ if $(h_1, h_2) \in H$ is an isomorphism. 
\end{enumerate}
\end{theorem}

Let $T_1$ and $T_2$ be transversals of $N_1$ in $G_1$ and $N_2$ in $G_2$ respectively. 
Let $\hat{\theta}:T_1 \rightarrow T_2$ be a map induced by $\theta$, where 
 $\hat{\theta}(t_1)=t_2$ if $\theta(N_1t_1) = N_2t_2$.  
One can show that $\hat{\theta}$ is a well-defined, and letting ${\mathcal{G}} = \{ (t_1, \hat{\theta}(t_1)) \mid t_1 \in T_1 \}$ be the graph of $\hat{\theta}$, 
we have $H = \langle {\mathcal{G}}, N_1 \times 1, 1 \times N_2 \rangle$.


For subdirect products of the direct product of more than two groups, we use an asymmetrical version of  \Cref{ori goursat}, where $\theta: G_1 \rightarrow G_2/N_2$ defined by $h_1 \mapsto N_2h_2$ for all $(h_1, h_2) \in H$ is a surjective homomorphism. Similarly, letting $\hat{\theta}:G_1 \rightarrow T_{2}$ be defined by $\hat{\theta}(g_1)=t_2$ if $\theta(g_1) = N_2t_2$ and ${\mathcal{G}}= \{ (g_1, \hat{\theta}(g_1)) \mid g_1 \in G_1 \}$ be the graph of $\hat{\theta}$, we get $H = \langle {\mathcal{G}}, 1 \times N_2 \rangle $. For more details, see \cite[Theorem~2.3]{GeneralisedGoursat}.

We will be applying Goursat's lemma to an intransitive group $H \leq S_n$, considered as a subdirect product of the direct product of its transitive constituents. 

Let $\Delta$ be a union of (some of) the $H$-orbits.  
For $h \in H$, we denote by $h|_{\Delta}$ the \emph{restriction} of $h$ onto $\Delta$, that is, the permutation in $\Sym(\Delta)$ such that $\alpha^{(h|_{\Delta})} = \alpha^h$ for all $\alpha \in \Delta$. 
We denote by $H|_{\Delta}$ the \emph{restriction} of $H$ onto $\Delta$, so $H|_{\Delta} = \{(h|_{\Delta}) \mid h \in H\}$. 
Let $K \leq \Sym(\Delta)$ and let $\Gamma$ be a set disjoint to $\Delta$. Denote by $1_{\Gamma} \times K$ the subgroup of $\Sym(\Delta \cup \Gamma)$ with support $\Delta$ such that $(1_{\Gamma} \times K)|_{\Delta} = K$.

Recall that we regard a direct product of groups that have disjoint supports as a subgroup of the symmetric group over the disjoint union of the supports of the factors. 
For the rest of the paper, we will use the following notation: 

\begin{notation} \label{set up H}
Let $H \leq S_n$. 
Fix an ordering $\Omega_1, \Omega_2, \ldots, \Omega_k$ on the $H$-orbits. 
For $1 \leq i \leq k$, let $G_i \coloneqq H|_{\Omega_i}$. 
We consider $H$ as a subdirect product of $G\coloneqq G_1 \times G_2 \times \ldots \times G_k \leq S_n$. \\
For $1 \leq i \leq k$, let $\proj_i: G \rightarrow G_i$ be defined by $h \mapsto h|_{\Omega_i}$. For a subset $I = \{I_1, I_2, \ldots, I_r \}$ of $\{1,2, \ldots,k \}$, let $\Proj_I: G \rightarrow H|_{\cup_{i \in I} \Omega_i}$ be defined by $h \mapsto h|_{\Omega_{I_1}} h|_{\Omega_{I_2}} \ldots h|_{\Omega_{I_r}}$. \\ 
For all $1 \leq i \leq k$, let $\Delta_i \coloneqq \bigcup \limits_{j \leq i}\Omega_j$.  
\end{notation}

Note that it is important that the $H$-orbits have to be in a fixed order. 
However, the choice of how we order them is unimportant. In our experiments, we chose to order the orbits by their smallest elements. 


By iteratively considering $\Proj_{\overline{i+1}}(H)$ as a subdirect product of $\Proj_{\overline{i}}(H) \times \proj_{i+1}(H)$, we get \Cref{goursat}, which gives the structure of intransitive groups and follows from \cite[Theorem~3.2]{GeneralisedGoursat}.

\begin{theorem}[{\cite{GeneralisedGoursat}}] \label{goursat}
Recall \Cref{set up H}. Then for all $1 \leq i \leq k-1$, the following hold. 
\begin{enumerate} 
    \item \label{Ni normal}
    Let $N_{i+1} \coloneqq  \proj_{i+1}(H_{(\Delta_i)})$. 
    Then $N_{i+1} \trianglelefteq G_{i+1}$. 
    \item \label{thetai surj hom} Let $\theta_i : \Proj_{\overline{i}}(H) \rightarrow G_{i+1}/N_{i+1}$ be defined by $\Proj_{\overline{i}}(h) \mapsto N_{i+1} \proj_{i+1}(h)$.
    Then $\theta_i$ is a surjective homomorphism. 
    \item \label{N and theta give next projn} Let $T_{i+1}$ be a transversal of $N_{i+1}$ in $G_{i+1}$ and
\begin{eqnarray*}
\varphi_i : \Proj_{\overline{i}}(H) &\rightarrow& \Sym( \Delta_{i+1} ) \\
\Proj_{\overline{i}}(h) &\mapsto& \Proj_{\overline{i}}(h) t \quad \text{ if $t \in T_{i+1}$ and $\theta_i(\Proj_{\overline{i}}(h)) =N_{i+1}t$}. 
\end{eqnarray*} 
Then $\Proj_{\overline{i+1}}(H) = \langle \mathrm{Im}(\varphi_i) , 1_{\Delta_i} \times N_{i+1} \rangle $. 
\end{enumerate}
\end{theorem}

In this paper, we will always require that $1 \in T_{i+1}$ for all $1 \leq i \leq k-1$. 
Note also that the $\varphi_i$ are not homomorphisms in general, but this is not a problem since we do not use it in our computation.

We end the section with an example, which will be a running example throughout the paper. 

\begin{example}\label{example: goursat} 
Let $x_1\coloneqq(1,2,3)(7,9,8)(10,12,11)$,  $x_2\coloneqq(4,5,6)(7,8,9)(10,11,12)$, \linebreak $x_3\coloneqq(5,6)(8,9)(11,12)$ and $ x_4\coloneqq (7,8,9)(10,11,12)$. 
Let $H \coloneqq \langle x_1, x_2, x_3, x_4 \rangle \leq S_{12}$.\\ 
Then $\Omega_1= \{ 1,2,3\}$, $\Omega_2= \{ 4,5,6 \}$,  $\Omega_3= \{ 7,8,9 \}$ and $ \Omega_4= \{ 10,11,12 \}$ are the orbits of $H$. 
For $1 \leq i \leq 4$, let $G_i = H|_{\Omega_i}$. Then $H$ is a subdirect product of $G_1 \times G_2 \times G_3 \times G_4$, where we regard the direct product as a subgroup of $S_{12}$. \\
For all $1 \leq i \leq 4$, let $\Delta_i \coloneqq \bigcup_{j \leq i}\Omega_j$. Then $\Proj_{\overline{i}}(H)= H|_{\Delta_i}$ for all $i$.
So $H|_{\Delta_2} = \langle (1,2,3),$ $ (4,5,6), $ $(5,6) \rangle$ and $H_{(\Delta_2)} = \langle (7,8,9)(10,11,12) \rangle$.  \\
Let the $N_i$ be as in \Cref{goursat}. Then $N_3 = \langle (7,8,9) \rangle$ is normal in $G_{3}$, and $N_4=1$.  \\
Let $\theta_2 : H|_{\Delta_2} \rightarrow G_3/N_3$ be as in \Cref{goursat}. Then 
$\theta((1,2,3)) = \theta(x_1|_{\Delta_2}) = N_3 x_1|_{\Omega_3} = N_3$. 
Similarly $\theta_2((4,5,6)) =  N_3$ and $\theta_2((5,6)) = N_3(8,9)$. 
So $\theta_2$ is surjective. \\
Let $T_3\coloneqq \{ (), (8,9) \}$ be a transversal of $N_3$ in $G_3$. 
Let $\varphi_2$ be as in \Cref{goursat}. 
Then $\varphi_2(H|_{\Delta_2}) = \langle (1,2,3), (4,5,6), (5,6)(8,9) \rangle$, and one could check that indeed $\langle \varphi_2(H|_{\Delta_2}) ,1_{\Delta_2} \times N_3 \rangle  = H|_{\Delta_3}$. 
\end{example}

}

\section{Disjoint direct product decomposition}
\label{theory of ddfd}

Recall \Cref{set up H}. In this section, we will first show that $H \leq S_n$ has a unique finest disjoint direct product decomposition.
Then in \Cref{3 lemmas}, we will see how the computation of disjoint direct product decompositions can be reduced to computing the $N_{i+1}$ and the kernels of the $\theta_i$ in \Cref{goursat}. We then show that these, in turn, can be efficiently computed in \Cref{section: ssgs}.

\begin{definition}
We say that two disjoint direct product decompositions $H =  H_1\times H_2\times \ldots \times H_{r}$ and $H =  K_1 \times  K_2 \times  \ldots \times K_{s}$ are \emph{equivalent}
if the sets of sets $\{\Supp(H_i) \mid 1 \leq i \leq r \}$ and $\{\Supp(K_i) \mid 1 \leq i \leq s \}$ are equal. 
\end{definition}

\begin{proposition}\label{finest DDPD is unique}
Up to equivalence, there is a unique finest disjoint direct product decomposition of $H$.
\end{proposition}

\begin{proof}
Aiming for a contradiction, let $H =  H_1\times H_2\times \ldots \times  H_{r}$ and $H =  K_1 \times  K_2 \times  \ldots \times K_{s}$ be two inequivalent finest disjoint direct product decompositions of $H$.
Since the supports of the disjoint direct factors form a partition of $\Supp(H)$, 
there exist $1 \leq i \leq r$ and $1 \leq j \leq s$ such that $\Supp(H_i) \neq \Supp(K_j)$ and $\Supp(H_i) \cap \Supp(K_j) \neq \emptyset$. 
Let $\Gamma \coloneqq \Supp(H_i)$ and $\Delta \coloneqq \Supp(K_j)$. 
We will show that 
$H_i = H_i|_{ \Gamma \backslash \Delta} \times  H_i|_{ \Gamma \cap \Delta}$ is a disjoint direct product decomposition of $H_i$, which contradicts the fact that $H =  H_1\times H_2 \times \ldots \times  H_{r}$ is a finest decomposition. 
Since $H_i$ is a subdirect product of $H_i|_{ \Gamma \backslash \Delta} \times  H_i|_{ \Gamma \cap \Delta}$, by the backward implication of  \Cref{DP iff stab gives full group} it suffices to show that  $H_i|_{\Gamma \cap \Delta} \times 1_{\Gamma \backslash \Delta} \leq H_i$.  
We do so by showing that for all $h_i \in H_i$, there exists $h_i' \in H_i$ such that  $h_i'|_{\Gamma \cap \Delta} = h_i|_{\Gamma \cap \Delta}$ and $h_i'|_{\Gamma \backslash \Delta} =1$. \\
Let $h_i \in H_i$. 
Let $\hat{h_i} \in S_n$ be such that $\hat{h_i}|_{\Gamma} = h_i$ and $\hat{h_i}|_{\overline{n} \backslash \Gamma} = 1$. 
By the forward implication of \Cref{DP iff stab gives full group}, $\hat{h_i} \in H$.
Similarly, since $K_j$ is a disjoint direct factor of $H$, 
there exists $h \in H$ such that $h|_{\overline{n} \backslash \Delta} = \hat{h_i}|_{\overline{n} \backslash \Delta}$ and $h|_{ \Delta} = 1$. 
Then $h' \coloneqq \hat{h_i} {h}^{-1}$ is an element of $H$ such that
\[
h'|_{\Gamma\cap \Delta} = (\hat{h_i}|_{\Gamma\cap \Delta}) ({h}^{-1}|_{\Gamma\cap \Delta}) =\hat{h_i}|_{\Gamma\cap \Delta} =  {h_i}|_{\Gamma\cap \Delta} \text{, } \] 
\[h'|_{\Gamma  \backslash \Delta} = (\hat{h_i}|_{\Gamma  \backslash \Delta}) ({h}^{-1}|_{\Gamma  \backslash \Delta})  = (\hat{h_i}|_{\Gamma  \backslash \Delta}) (\hat{h_i}^{-1}|_{\Gamma  \backslash \Delta}) =1, \]
\[\text{ and } h'|_{\overline{n} \backslash \Gamma} = h'|_{\Delta \backslash \Gamma} h'|_{\overline{n} \backslash (\Delta \cup \Gamma)} = (\hat{h_i}|_{\Delta \backslash \Gamma} h^{-1}|_{\Delta \backslash \Gamma}) (\hat{h_i}|_{\overline{n} \backslash (\Delta \cup \Gamma)} \hat{h_i}^{-1}|_{\overline{n} \backslash (\Delta \cup \Gamma)} ) =1 .\]
Therefore, $h_i' := h'|_{\Gamma}$ is an element of $H_i$ such that $h_i'|_{\Gamma \cap \Delta} = h_i|_{\Gamma \cap \Delta}$ and $h_i'|_{\Gamma \backslash \Delta} =1$.
\end{proof}



\subsection{Computing the disjoint direct product decomposition}
\label{3 lemmas}

Recall \Cref{set up H}. 
We will compute the finest disjoint direct product decomposition of $H$ iteratively by computing the finest disjoint direct product decomposition of $\Proj_{\overline{i}}(H)$ for $1 \leq i \leq k$. 
In this subsection, we show for $1 \leq i < k$, how we can compute the finest disjoint direct product decomposition of $\Proj_{\overline{i+1}}(H)$ using the finest disjoint direct product decomposition of $\Proj_{\overline{i}}(H)$ and the groups $N_{i+1}$ and homomorphisms $\theta_i$ from \Cref{goursat}. 

Since the support of each disjoint direct factor of a group is a union of (some of) its orbits, 
we will be computing certain partitions of $\overline{i}$ for each $1 \leq i \leq k$. We will denote an unordered partition $\mathcal{P}$ of a set $\Gamma$ by $\langle C_1 \mid C_2 \mid \ldots \mid C_r \rangle$, where the $C_i$ are sets with disjoint intersections, called the \emph{cells} of $\mathcal{P}$, and the union of the $C_i$ is $\Gamma$. 

\begin{definition} \label{defin: i partition}
For $1 \leq i \leq k$, let $\mathcal{P}_{i} = \langle C_1 \mid C_2 \mid \ldots \mid C_r \rangle $ be the (unordered) partition of $\overline{i}$ consisting of cells $C_j \subseteq \overline{i}$ for $1 \leq j \leq r$ such that $\Proj_{\overline{i}}(H) = \Proj_{C_1}(H) \times \Proj_{C_2}(H) \times \ldots \times \Proj_{C_r}(H)$ is the finest disjoint direct product decomposition of $\Proj_{\overline{i}}(H)$.
\end{definition} 



Observe that trivially, $\mathcal{P}_1 = \langle \{1 \} \rangle$.
\Cref{finest ddfd i+1} describes how we can compute $\mathcal{P}_{i+1}$ using $\mathcal{P}_{i}$ for $1 \leq i <k$. 
To simplify notation, from now on, for subsets $I,J \subseteq \overline{k}$ and for $h \in H$, we denote by $\Proj_I(h) \times 1_{J}$ the permutation $h' \in \Sym(\Supp(\Proj_{I \cup J}(H)))$ such that $\Proj_I(h') = \Proj_I(h)$ and $ \Proj_{J}(h')=1$. Similarly, for $K \subseteq H$, we denote by $\Proj_I(K) \times 1_{J}$ the set 
$\{ \Proj_I(h) \times 1_J \mid h \in K \}$. 

\begin{proposition} \label{finest ddfd i+1}\label{finest ddfd i+1 - partn ver}
Let $1 \leq i < k$. 
Let $\mathcal{P}_{i} = \langle C_1 \mid C_2 \mid \ldots \mid C_r \rangle $ be as in \Cref{defin: i partition} and 
let $\theta_i$ be as in \Cref{goursat}.
Let $S\coloneqq \{ C_j \mid  1 \leq j \leq r , \, \Proj_{C_j}(H) \times 1_{\overline{i}\backslash C_j} \not \subseteq \mathrm{Ker}(\theta_i)\}$ and 
let $C = \bigcup_{C_j \in S} C_j \cup \{i+1 \}$.~\footnote{Note that $S$ depends in $i$. Since it is clear from the context which $i$ the set $S$ refers to, we omit the reference to simplify notation. }
Then 
\begin{equation} \label{eqn: finest ddfd of i+1}
 \Proj_{\overline{i+1}}(H)=  \Proj_C(H) \times \prod_{C_j \not \in S} \Proj_{C_j}(H)
\end{equation}
is the finest disjoint direct product decomposition of $\Proj_{\overline{i+1}}(H)$. \\
Hence the partition $\mathcal{P}_{i+1}$ from \Cref{defin: i partition} is the partition of $\overline{i+1}$ consisting of cell $C$ and all other cells $C_j$ of $\mathcal{P}_i$ such that $C_j \not \in S$. 
\end{proposition}

\begin{proof}
We will first show that \Cref{eqn: finest ddfd of i+1} is a disjoint direct product decomposition of $\Proj_{\overline{i+1}}(H)$, and then we will show that it is the \emph{finest} disjoint direct product decomposition.
The statement on $\mathcal{P}_{i+1}$ will then follow from \Cref{defin: i partition}. \\
Since the factors in \Cref{eqn: finest ddfd of i+1} move disjoint sets of points, we show that \Cref{eqn: finest ddfd of i+1} is a disjoint direct product decomposition of $\Proj_{\overline{i+1}}(H)$, by showing that it gives a direct product decomposition of $\Proj_{\overline{i+1}}(H)$. 
Observe that $\Proj_{\overline{i+1}}(H)$ is a subdirect product of $\Proj_C(H) \times \prod_{C_j \not \in S} \Proj_{C_j}(H)$.
Then by the backward implication of \Cref{DP iff stab gives full group}, it suffices to show that $1_C \times  \prod_{C_j \not \in S} \Proj_{C_j}(H) \leq \Proj_{\overline{i+1}}(H)$. 
We will do so by showing that 
$\Proj_{C_j}(H) \times 1_{\overline{i+1} \backslash C_j} \leq \Proj_{\overline{i+1}}(H)$ for all cells $C_j$ of $\mathcal{P}_i$ such that $C_j \not \in S$, then from here it follows that \Cref{eqn: finest ddfd of i+1} is a disjoint direct product decomposition of $\Proj_{\overline{i+1}}(H)$.  \\
Let $C_j \not \in S$. Then $\theta_i(\Proj_{C_j}(H) \times 1_{\overline{i} \backslash C_j}) = N_{i+1}$. 
Let $\varphi_i$ be as in \Cref{goursat}. Then $\varphi_i(\Proj_{C_j}(H) \times 1_{\overline{i} \backslash C_j}) = \Proj_{C_j}(H) \times 1_{\overline{i+1} \backslash C_j}  \leq \Proj_{\overline{i+1}}(H)$. 
Hence, \Cref{eqn: finest ddfd of i+1} gives a disjoint direct product decomposition of $\Proj_{\overline{i+1}}(H)$. \\
We will now show that \Cref{eqn: finest ddfd of i+1} is the \emph{finest} disjoint direct product decomposition.
As $C_j \not \in S$ are cells of $\mathcal{P}_i$, the groups $\Proj_{C_j}(H)$ for $C_j \not \in S$ are d.d.p.\ indecomposable, so it remains to show that $\Proj_C(H)$ is d.d.p.\ indecomposable. 
Observe that for $C_j \not \in S$, since $\Proj_{C_j}(H)$ is a finest disjoint direct factor of $\Proj_{\overline{i+1}}(H)$, each $C_j \not \in S$ is a cell of $\mathcal{P}_{i+1}$.
We proceed as follows. 
We first show that $C$ is the union of a subset of
$U \coloneqq \{ C_1, C_2, \ldots, C_r, \{i+1\} \}$ and then show that $C_j \in S$ is in the same cell of $\mathcal{P}_{i+1}$ with $\{ i+1\}$, from which we deduce that $C$ is a cell of $\mathcal{P}_{i+1}$ and so $\Proj_C(H)$ is d.d.p.\ indecomposable. \\
To prove the first claim, we show that for all $u \in U$ such that $u \cap C \neq \emptyset$, we have $u \subseteq C$. This is trivially true for $u = \{i+1\}$. 
Let $C_j$ be a cell of $\mathcal{P}_{i}$ such that $C_j \cap C \neq \emptyset$. 
We have seen that \Cref{eqn: finest ddfd of i+1} is a disjoint direct product decomposition, thus the projection $\Proj_C(H)$ is a disjoint direct factor of $\Proj_{\overline{i+1}}(H)$, so $\Proj_{C}(H) \times 1_{\overline{i+1} \backslash C} \leq \Proj_{\overline{i+1}}(H)$. Applying $P_{C_j}$ on both sides yields $\Proj_{C_j \cap C}(H) \times 1_{C_j \backslash C} \leq \Proj_{C_j}(H)$. Since $\Proj_{C_j}(H)$ is d.d.p.\ indecomposable, $C_j \cap C= C_j$, so $C_j \subseteq C$ and therefore $C$ is the union of a subset of $U$.  \\
Lastly, we show that each $C_j \in S$ is in the cell of $\mathcal{P}_{i+1}$ containing $i+1$. 
Let $C_j \in S$. Aiming for a contradiction, suppose that $C_j$ and $i+1$ are in different cells of $\mathcal{P}_{i+1}$. 
Let $\varphi_i: \Proj_{\overline{i}}(H) \rightarrow \Sym(\Delta_{i+1})$ be as in \Cref{goursat} and consider $L:= \varphi_i(\Proj_{C_j}(H) \times 1_{\overline{i}\backslash C_j}) \subseteq \Proj_{\overline{i+1}}(H)$. 
Let $D$ be a cell of $\mathcal{P}_{i+1}$ containing $C_j$.
Then $\Proj_D(L) \times 1_{\overline{i+1}\backslash D} \leq \Proj_D(H) \times 1_{\overline{i+1}\backslash D} $ is contained in $ \Proj_{\overline{i+1}}(H)$. 
Since  $i+1 \not \in D$, it follows that
\[
\Proj_{\overline{i}}(L) \times 1_{i+1} 
= \Proj_D(L) \times 1_{\overline{i+1} \backslash D}  \leq \Proj_{\overline{i+1}}(H). 
\]
Now since both $L$ and $\Proj_{\overline{i}}(L) \times 1_{i+1}$ are contained in $\Proj_{\overline{i+1}}(H)$, the set $1_{\overline{i}} \times \proj_{i+1}(L)$ is also contained in $\Proj_{\overline{i+1}}(H)$.
Therefore $\proj_{{i+1}}(L) \subseteq N_{i+1}$ and hence $\theta_i(\Proj_{C_j}(H) \times 1_{\overline{i} \backslash C_j}) = N_{i+1}$, a contradiction to the fact that $C_j \in S$.
\end{proof}

\begin{example}[running example]  \label{example: main prop} 
We return to \Cref{example: goursat}. 
First observe that \linebreak $\Proj_{\overline{2}}(H) = \langle (1,2,3),(4,5,6),(5,6) \rangle=  \proj_{1}(H) \times \proj_{2}(H)$, so $\mathcal{P}_2 = \langle \{1\} \mid  \{2\} \rangle$. 
We have seen from \Cref{example: goursat} that $\theta((1,2,3)) = \theta_2((4,5,6)) = N_3$ and $\theta_2((5,6)) = N_3(8,9)$. 
So $\proj_{1}(H) \times 1_{\overline{3}\backslash \{1\} } = \langle (1,2,3) \rangle$ is contained in $\mathrm{Ker}(\theta_i)$ while $\proj_{2}(H) \times 1_{\overline{3}\backslash \{2\} } = \langle (4,5,6), (5,6) \rangle$ is not contained in $\mathrm{Ker}(\theta_i)$.
Hence $\mathcal{P}_3 = \langle \{1\} \mid \{2,3\} \rangle$ and one can check that indeed $\Proj_{\overline{3}}(H) = \proj_{1}(H) \times \Proj_{\{2,3\}}(H)$.
\end{example}

\Cref{finest ddfd i+1} will be used as the core of our algorithm for finding the finest disjoint direct factor decomposition in \Cref{section: algorithm}.

\subsection{Orbit-ordered base and separable strong generating set}
\label{section: ssgs}

In this subsection, fix $1 \leq i <k$ and let $N_{i+1}$, $\theta_i$ and $\varphi_i$ be as in \Cref{goursat}. 
We will see how we can use some fundamental data structures associated to permutation groups to compute the $N_{i+1}$ and find the cells $C_j$ of $\mathcal{P}_i$ such that $\Proj_{C_j}(H) \times 1_{\overline{i}\backslash C_j} \not \subseteq \mathrm{Ker}(\theta_i)$.

\ifthenelse{\equal{\thesis}{1}}{
Recall the definitions of base and strong generating sets from \Cref{def: base ,def: sgs}. 
In this chapter, we will write bases with square brackets to differentiate between bases and permutations. Recall the $\Delta_i$ from \Cref{set up H}.
We will be using bases that are orbit-ordered:}{
We will be using a base and strong generating set, which is used in many permutation group algorithms \cite{sims1970, sims1971a}. For more information on how to find bases and strong generating sets, see \cite{handbook, seress}.

\begin{definition} \label{defn: base}\label{defn: sgs}
A \emph{base} of $H \leq S_n$ is a sequence $B= [ \beta_1, \beta_2, \ldots, \beta_m ]$ of points in $\overline{n}$ for some $m \in \mathds{Z}^+$ such that $H_{(\beta_1, \beta_2, \ldots, \beta_m)} = 1$. \\
A base defines a subgroup chain $H^{[1]} \geq H^{[2]} \geq \ldots \geq H^{[m+1]}$ called a \emph{stabiliser chain}, where $H^{[1]} \coloneqq H$ and $H^{[i]} \coloneqq H_{(\beta_1, \beta_2, \ldots, \beta_{i-1})}$ for $2 \leq i \leq m+1$. \\
A base $B$ is \emph{non-redundant} if for all $1 \leq i \leq m$, we have 
$H^{[i+1]} \lneqq H^{[i]}$.
Otherwise, $B$ is said to be \emph{redundant}. \\
A \emph{strong generating set} $X$ of $H$ relative to $B$ is a generating set of $H$ where
$H^{[i]} = \langle x \in X \mid x \in H^{[i]} \rangle$ for all $1 \leq i \leq m+1$. 
\end{definition}
}

\ifthenelse{\equal{\thesis}{1}}{}{Recall the $\Delta_i$ from \Cref{set up H}. We will be using bases that are orbit-ordered: }

\begin{definition} \label{defn: orbit ordered base}
A base $B\coloneqq [\beta_1, \beta_2, \ldots, \beta_m]$ of $H \leq S_n$ is \textit{orbit-ordered} with respect to the ordering $\Omega_1, \Omega_2, \ldots, \Omega_k$ of the $H$-orbits if 
there exist $1 \leq j_1 \leq j_2 \leq \ldots \leq j_k \leq m$ such that 
for all $1 \leq i \leq k$, we have $H_{(\beta_1, \beta_2, \ldots, \beta_{j_i})} = H_{(\Delta_{i})}$.
\end{definition}

\begin{remark} \label{concat orbs gives orbit ordered base}
Recall from \Cref{set up H} that we fixed an ordering $\Omega_1,$ $\Omega_2,$ $\ldots,$ $\Omega_k$ of the $H$-orbits.
Then the concatenation of $\Omega_1,\Omega_2, \ldots,\Omega_k$ is a (redundant) orbit-ordered base of $H$.
\end{remark}

We can compute $N_{i+1}$ from a strong generating set of $H$ relative to an orbit-ordered base. 


\begin{lemma} \label{can get N from sgs}
Let $B\coloneqq [\beta_1, \beta_2, \ldots, \beta_m]$ be an orbit-ordered base of $H \leq S_n$ with respect to the ordering $\Omega_1, \Omega_2, \ldots, \Omega_k$.
Let $X$ be a strong generating set of $H$ with respect to $B$.
Then $N_{i+1} = \langle \proj_{i+1}(x)  \mid x \in X \cap H_{(\Delta_{i})} \rangle$ for all $1 \leq i <k$. 
\end{lemma}

\begin{proof}
Since $B$ is orbit-ordered, there exists $1 \leq j \leq m$ such that $H_{(\beta_1, \beta_2, \ldots, \beta_j)} = H_{(\Delta_{i})}$.
As $X$ is a strong generating set relative to $B$, we have $H_{(\Delta_{i})} = \langle X \cap H_{(\Delta_{i})} \rangle$.
The result follows since $N_{i+1} = \proj_{i+1}(H_{(\Delta_{i})})$.
\end{proof}

Recall that we compute the $\mathcal{P}_{i}$ iteratively. At the $i$-th iterative step, we will produce a strong generating set for $H$ that is $i$-separable:  

\begin{definition} \label{defn: separable}
Let $\mathcal{P}_{i}$ be as in \Cref{defin: i partition}. 
A strong generating set $X$ of $H$ with respect to an orbit-ordered base $B$ is \emph{${i}$-separable} if 
for all $x \in X$ with $\Proj_{\overline{i}}(x) \neq 1$, there exists a unique cell $C_j$ of $\mathcal{P}_i$ such that
$\Proj_{\overline{i}}(x) \in \Proj_{C_j}(H) \times 1_{\overline{i} \backslash C_j}$.\\
We say that $X$ is \emph{separable} if it is $k$-separable, where $k$ is the number of $H$-orbits.
\end{definition}

Note that since we have fixed an ordering of the $H$-orbits, by \Cref{finest DDPD is unique}, the $\mathcal{P}_i$ are unique, so ${i}$-separability depends only on $i$. 

\begin{example}[running example] \label{example: i separable} 
\ifthenelse{\equal{\thesis}{1}}{Let $x_1, x_2, x_3, x_4$ and $H$ be as in \Cref{example: goursat}.}{Recall \Cref{example: goursat}, where $x_1\coloneqq(1,2,3)(7,9,8)(10,12,11)$,  $x_2\coloneqq(4,5,6)(7,8,9)(10,11,12)$,  $x_3\coloneqq(5,6)(8,9)(11,12)$ and $ x_4\coloneqq (7,8,9)(10,11,12)$.} 
We order the $H$-orbits by their smallest elements, so we have an ordering $\Omega_1, \Omega_2, \Omega_3, \Omega_4$. 
Then $B \coloneqq [1,4,5,7]$ is an orbit-ordered base since 
$H_{(\Delta_1)} = H_{(1)}$,  $H_{(\Delta_2)} = H_{(1,4,5)}$,  $H_{(\Delta_3)} = H_{(\Delta_4)} = H_{(1,4,5,7)}$.
The generating set $X\coloneqq\{x_1, x_2, x_3, x_4 \}$ is a strong generating set of $H$ with respect to the base $B$. 
So $\Proj_{\overline{2}}(H) = \langle (1,2,3), (4,5,6), (5,6) \rangle$ and hence 
$\mathcal{P}_2 = \langle \{1\} \mid \{2\} \rangle$. 
Then $X$ is $2$-separable since  
\[ 
\Proj_{\overline{2}}(x_1) \in \proj_1(H) \times 1_2 \quad \text{and} \quad 
\Proj_{\overline{2}}(x_2), \Proj_{\overline{2}}(x_3) \in 1_1 \times \proj_2(H) \quad \text{and} \quad \Proj_{\overline{2}}(x_4)=1.
\]
Similarly, we can show that $\mathcal{P}_3 = \langle \{1\} \mid \{ 2, 3\} \rangle$.
The set $X$ is not $3$-separable since
$\proj_1(x_1)$ and $\Proj_{\{ 2,3 \}}(x_1)$ are both non-trivial.
\end{example}

For each cell $C_j$ of $\mathcal{P}_i$, 
we may compute $\theta_i( \Proj_{C_j}(H) \times 1_{\overline{i} \backslash C_j})$ using an ${i}$-separable strong generating set of $H$:

\begin{lemma} \label{get theta from sgs}
Let $X$ be an ${i}$-separable strong generating set of $H$. Then for each cell $C_j$ of $\mathcal{P}_i$, we have
\[\theta_i( \Proj_{C_j}(H) \times 1_{\overline{i} \backslash C_j}) = \langle N_{i+1}\proj_{i+1}(x) \mid x \in X \text{ and } \Proj_{C_j}(x) \neq 1 \rangle.\]
\end{lemma}

\begin{proof}
Since $X$ is ${i}$-separable, $\Proj_{C_j}(H) \times 1_{\overline{i} \backslash C_j}$ is generated by the $\Proj_{\overline{i}}(x)$ where $x \in X$ such that $\Proj_{C_j}(x) \neq 1$. 
So $\theta_i( \Proj_{C_j}(H) \times 1_{\overline{i} \backslash C_j}) = \langle N_{i+1}\proj_{i+1}(x) \mid x \in X \text{ and } \Proj_{C_j}(x) \neq 1 \rangle$. 
\end{proof}

\ifthenelse{\equal{\thesis}{1}}{
Recall the sifting procedure and the definition of siftees from \Cref{sifting procedure}. Note that $g_s$ is a product of $g$ and an element of $H$: 
}{
In our algorithm, we will be using the \emph{sifting procedure} to compute an $(i+1)$-separable strong generating set from an $i$-separable strong generating set.
The sifting of $g \in S_n$ by $H$ attempts to write $g$ as an element of $H$, and can be used for membership testing. For more details on sifting, please refer to \cite[Chapter~4]{seress}. 

\begin{definition}\label{sifting procedure}
Let $B = [\beta_1, \beta_2, \ldots, \beta_m]$ be a base of $H$, and let $H^{[i]}$ for $1 \leq i \leq m+1$ be the stabilisers in the stabiliser chain defined by $B$, as in \Cref{defn: base}.
For each $1 \leq i \leq  m$, let $R_i$ be a transversal of $H^{[i+1]}$ in $H^{[i]}$. 
Let $g \in S_n$ be given. 
The so-called \emph{sifting of $g$ by $H$} works as follows. We
initialise $g_1 \coloneqq g$. 
For $1 \leq i \leq m$, we recursively find $r_i \in R_i$ such that $\beta_i^{r_i} = \beta_i^{g_i}$, and set $g_{i+1} \coloneqq g_{i} r_{i}^{-1}$. 
The procedure terminates when either
\begin{enumerate}
    \item $2 \leq s \leq m$ and there are no $r_s \in R_s$ such that $\beta_s^{r_s} = \beta_s^{g_s}$, or
    \item $s=m+1$ and we have computed $g_{s}. $ 
\end{enumerate}
In both cases, $g_{s}$ is a \emph{siftee} of $g$ by $H$. 
\end{definition}

We may conclude that $g \in H$ if $g_{m+1}= g r_1^{-1} r_2^{-1} \ldots r_m^{-1} =1$. 
Observe that for each $1 \leq i \leq  s-1$, the permutation $g_{i+1}$ fixes $\beta_i$. 
Then as $\beta_i$ is fixed by all $R_j$ for $j \geq i$, the siftee $g_s$ also fixes $\beta_i$. 
Note also that $g_s$ is a product of $g$ and an element of $H$: }

\begin{lemma}\label{props of siftee by stab}
Let $g \in S_n$ and let $g'$ be a \emph{siftee} of $g$ by $H$. Then there exists $h \in H$ such that $g = g'h$.
\end{lemma}

\begin{proof}
Let $B$ and the $R_i$ be as in \Cref{sifting procedure}. Then there exist $1 \leq s \leq m+1$ and $r_i \in R_i$ for all $1 \leq i \leq s-1$ such that $g' = g r_1^{-1} r_2^{-1} \ldots r_{s-1}^{-1}$. 
Since the $R_i \subseteq H$, the result follows. 
\end{proof}

\begin{example} [sifting]
Let $X = \{ (1,2,3,4,5), (2,5)(3,4) \}$ and $H \coloneqq \langle X \rangle$. 
Then $X$ is a strong generating set of $H$ with respect to base $B\coloneqq [1,2]$.  
Let $R_1 \coloneqq \{ 1, (1,5,4,3,2),$ $ (1,4,2,5,3),$ $(1,2)(3,5),$ $(1,3,5,2,4) \}$ and $R_2 \coloneqq \{1,  (2,5)(3,4)\}$ be the transversals for $H_{(1)}$ in $H$ and for $H_{(1,2)}$ in $H_{(1)}$ respectively. 
Consider sifting $g = (1,2,4,5)$ by $H$.  
Initialise $g_1 \coloneqq g$. Then $r_1 \coloneqq (1,2)(3,5)$ is an element of $R_1$ mapping $1$ to $1^{g_1} = 2$. So $g_2 = g_1 r_1^{-1} =(2,4,3,5)$. 
Now there is no $r_2 \in R_2$ mapping $2$ to $2^{g_2} = 4$. Therefore we get a siftee $(2,4,3,5)$. \\
If we have $(1,2,3,4,5)$ in $R_1$ instead of $(1,2)(3,5)$, then we get siftee $g (1,2,3,4,5)^{-1} = (2,3)$. 
\end{example}

The siftee of $g \in S_n$ by $H$ obtained from a sifting procedure is not unique: it depends on the choice of the base $B$ and the transversals $R_i$ associated with the stabiliser chain defined by $B$.


\ifthenelse{\equal{\thesis}{1}}{

Recall that a base $B=[\beta_1, \beta_2, \ldots, \beta_m]$ of $H$ defines a subgroup chain $H^{[1]} \geq H^{[2]} \geq \ldots \geq H^{[m+1]}$ called a \emph{stabiliser chain}, where $H^{[1]} \coloneqq H$ and $H^{[i]} \coloneqq H_{(\beta_1, \beta_2, \ldots, \beta_{i-1})}$ for $2 \leq i \leq m+1$.

\begin{remark} \label{sift by reusing stab chain}
In our algorithm, we will be sifting elements of $H$ by stabilisers $H_{(\Delta_{i})}$. 
Let $B\coloneqq[\beta_1, \beta_2, \ldots, \beta_m]$ be an orbit-ordered base of $H$ with respect to the ordering $\Omega_1, \Omega_2, \ldots, \Omega_k$. Then there exists $1 \leq j \leq m$ such that $[\beta_{j+1}, \beta_{j+2}, \ldots, \beta_m]$ is a base of $H_{(\Delta_{i})}$, with associated stabiliser chain $H^{[j+1]} \geq H^{[j+2]} \geq \ldots \geq H^{[m+1]}$, which is a subchain of the stabiliser chain of $H$ defined by $B$.  
So whenever we sift an element $h \in H$ by a stabiliser $H_{(\Delta_{i})}$, we may sift using the stabiliser chain defined by an orbit-ordered base $B$ of $H$.
\end{remark}

}{
In our algorithm, we will be sifting elements of $H$ by stabilisers $H_{(\Delta_{i})}$. Let $B\coloneqq(\beta_1,  \ldots, \beta_m)$ be an orbit-ordered base of $H$ with respect to the ordering $\Omega_1, \Omega_2, \ldots, \Omega_k$. 
Then there exists $1 \leq j \leq m$ such that $(\beta_{j+1}, \beta_{j+2}, \ldots, \beta_m)$ is a base of $H_{(\Delta_{i})}$ (where $j$ here is the $j_i$ from \Cref{defn: orbit ordered base}), with associated stabiliser chain $H^{[j+1]} \geq H^{[j+2]} \geq \ldots \geq H^{[m+1]}$, which is a subchain of the stabiliser chain of $H$ defined by $B$.  
Whenever we sift an element $h \in H$ by a stabiliser $H_{(\Delta_{i})}$, we sift using the partial stabiliser chain defined by an orbit-ordered base $B$ of $H$. 

\begin{definition} \label{sift by reusing stab chain}
Let $B\coloneqq(\beta_1, \beta_2, \ldots, \beta_m)$ be an orbit-ordered base of $H$ and let $1 \leq j \leq m$ such that $(\beta_{j+1}, \beta_{j+2}, \ldots, \beta_m)$ is a base of $H_{(\Delta_{i})}$. 
For $h \in H$, a \emph{siftee of $h$ by $H_{(\Delta_{i})}$}
is a siftee obtained by sifting $h$ using the stabiliser chain defined by the base $(\beta_{j+1}, \beta_{j+2}, \ldots, \beta_m)$. 
\end{definition}
}

For $x \in X$, we will be deciding whether $\proj_{i+1}(x) \in N_{i+1}$ by considering a siftee of $x$ by $H_{(\Delta_i)}$. 

\begin{lemma}\label{image of theta by siftee}
Let $B$ be an orbit-ordered base of $H \leq S_n$ with respect to the $H$-orbits ordering $\Omega_1, \Omega_2, \ldots, \Omega_k$. 
Let $x \in H$ and $x'$ be a siftee of $x$ by $H_{(\Delta_i)}$, where we sift $x$ as in \Cref{sift by reusing stab chain}.
Then $\proj_{i+1}(x) \in N_{i+1}$ if and only if $\proj_{i+1}(x') = 1$. 
\end{lemma}

\begin{proof}
If $\proj_{i+1}(x) \in N_{i+1}$, by \Cref{props of siftee by stab}, there exists $h \in H_{(\Delta_i)}$ such that $x=x'h$. 
So $\proj_{i+1}(x) = \proj_{i+1}(x') \proj_{i+1}(h) \in N_{i+1}$.  \\
Suppose now that $\proj_{i+1}(x') = 1$. Since $B$ is orbit-ordered, there exist $1 \leq t \leq u \leq m$ such that $H_{(\beta_1, \beta_2, \ldots ,\beta_t)} = H_{(\Delta_i)}$ and $H_{(\beta_1, \beta_2, \ldots, \beta_u)} = H_{(\Delta_{i+1})}$.  
So $[\beta_{t+1}, \beta_{t+2}, \ldots, \beta_{u}]$ is a base of $\proj_{i+1}(H_{(\Delta_i)}) = N_{i+1}$. 
Since $\proj_{i+1}(x) \in N_{i+1}$, sifting $x$ does not terminate until after we have considered the image of $\beta_{u}$. 
More specifically, there exists $u \leq s \leq m+1$ and $r_i \in R_i$ for all $1 \leq i \leq s-1$ such that $x' = x r_{t+1}^{-1} r_{t+2}^{-1} \ldots r_{s-1}^{-1}$. 
Hence $x'$ fixes $[\beta_{t+1}, \beta_{t+1}, \ldots, \beta_{u}]$. 
Therefore $\rho_{i+1}(x')$ is an element of $N_{i+1}$ fixing the base of $N_{i+1}$, so $\proj_{i+1}(x')=1$. 
\end{proof}

\section{Algorithm}
\label{section: algorithm}


In this section, we present an algorithm to compute the finest disjoint direct product decomposition of a given permutation group $H = \langle X \rangle \leq S_n$ and show that it has polynomial time complexity in terms of $n \cdot |X|$.

Recall \Cref{set up H}. 
For $1 \leq i \leq k$, let $\mathcal{P}_i$ be as in \Cref{defin: i partition}. 
As the base case, we begin with $\mathcal{P}_1 \coloneqq \langle \{1\} \rangle$ and 
we will compute $\mathcal{P}_{i+1}$ iteratively using \Cref{finest ddfd i+1 - partn ver}.

Our algorithm to compute the finest disjoint direct product decomposition is presented in \Cref{algm: DDPD overall}. 
This algorithm computes a base and an initial strong generating set for \(H\) and then calls \Cref{algm: DDPD} repeatedly to calculate the \(\mathcal{P}_{i}\).

\begin{algorithm}[ht]
\caption{Finding the finest disjoint direct product decomposition of $H$} 
\label{algm: DDPD overall}
\textbf{Input:} $H \leq S_n$, given by a generating set $X$ \\
\textbf{Output:} $\mathcal{P}_k$, where $k$ is the number of $H$-orbits
\begin{algorithmic}[1]
    \State Fix an ordering $\mathcal{O}$ of the $H$-orbits $\Omega_1, \Omega_2, \ldots, \Omega_k$ 
    \State Find an orbit-ordered base $B$ of $H$ with respect to the ordering of $H$-orbits
    \State Compute a strong generating set $X_1$ of $H$ relative to $B$ 
    \State Initialise $\mathcal{P}_1 = \langle \{ 1 \} \rangle$
    \For{$i \in [1..k-1]$}
        \State $X_{i+1},\mathcal{P}_{i+1} \gets DDPD(i, B, X_i, \mathcal{P}_{i}, \mathcal{O})$
    \EndFor    
    \State \Return $\mathcal{P}_{k}$
\end{algorithmic}
\end{algorithm}

\begin{algorithm}[h!]
\caption{Finding the finest disjoint direct product decomposition of $\Proj_{\overline{i+1}}(H)$} \label{algm: DDPD}
\textbf{Input:} Integer $1 \leq i \leq k$, base $B$ of $H$, partition $\mathcal{P}_i= \langle C_1 \mid C_2 \mid \ldots \mid C_r \rangle$ of $\{1,2, \ldots , k\}$ where $k$ is the number of $H$-orbits, an $i$-separable generating set $X_i$ of $H$, and an ordering $\mathcal{O}$ of the $H$-orbits \\
\textbf{Output:} $\mathcal{P}_{i+1}$ and an $(i+1)$-separable strong generating set $X_{i+1}$ of $H$
\begin{algorithmic}[1]
    \Procedure{DDPD}{$i$, $B$, $X_i$, $\mathcal{P}_i$, $\mathcal{O}$}
        \State $S \gets \{\}$
        \State $X_{i+1} \gets \{ \}$
        \For{$x \in X_i$} 
            \If {$x \not \in H_{(\Delta_i)}$}
                \State Find cell $C_j$ of $\mathcal{P}_i$ such that $\Proj_{\overline{i}}(x) \in \Proj_{C_j}(H) \times 1_{\overline{i} \backslash C_j}$
                \State $x' \gets$ siftee of $x$ by $H_{(\Delta_i)}$ using the stabiliser chain defined by $B$  \newline \hspace*{3em}  \Comment{see \Cref{sift by reusing stab chain}}
                \State Add $x'$ to $X_{i+1}$
                \If {$\proj_{i+1}(x') \neq 1$} 
                    \State Add $C_j$ to $S$
                \EndIf
            \Else    
                \State Add $x$ to $X_{i+1}$
            \EndIf        
        \EndFor
            \State $C \gets ( \bigcup \limits_{C_j \in S} C_j) \cup \{i+1\}$
            \State $\mathcal{P}_{i+1} \gets $ the partition consisting of cell $C$ and all other cells $C_j$ of $\mathcal{P}_i$ such that $C_j \not \in S$
        \State \Return $ X_{i+1}, \mathcal{P}_{i+1}$
    \EndProcedure
\end{algorithmic}
\end{algorithm}

We give an example for \Cref{algm: DDPD} here. 

\begin{example} [running example]
Consider $H$ in \Cref{example: i separable}.
We have seen that $\mathcal{P}_2 = \langle \{1\} \mid \{2\} \rangle$ and $X$ is 2-separable. 
To compute $\mathcal{P}_3$, initialise $S \coloneqq \{ \}$.
Observe that $\Proj_{\overline{2}}(x_1) \in \Proj_1(H) \times 1_2$ and $x_1 \not \in H_{(\Delta_2)}$. Sifting $x_1$ by $H_{(\Delta_2)}$ gives us a siftee $(1,2,3)$.
Similarly, $\Proj_{\overline{2}}(x_2) \in \Proj_2(H) \times 1_1$ and $x_2 \not \in H_{(\Delta_2)}$ with a siftee $(4,5,6)$. 
Now $x_3 \in \Proj_2(H) \times 1_1$ and $x_3 \not \in H_{(\Delta_2)}$ has a siftee $(5,6)(8,9)(11,12)$, which has a non-trivial projection on $\{ 7,8,9\}$. So we add $\{2\} $ to $S$. 
Since $x_4 \in H_{(\Delta_2)}$ we add $x_4$ to $X_3$.
Therefore,  $X_3= \{ (1,2,3), (4,5,6), (5,6)(8,9)(11,12),$ $ (7,8,9)(10,11,12)\}$ and $\mathcal{P}_3 = \langle \{1\} \mid \{2,3\} \rangle$. 
\end{example}

\begin{lemma}
For $1 \leq i \leq k$, the sets $X_{i}$ computed in \Cref{algm: DDPD overall,algm: DDPD} are strong generating sets of $H$ with respect to the base $B=[\beta_1,\beta_2, \ldots,\beta_m]$ defined on line 2 of \Cref{algm: DDPD overall}.  
\end{lemma}

\begin{proof}
We proceed by induction on $i$. 
Clearly $X_1$ is a strong generating set of $H$ with respect to $B$. 
Suppose that $X_i$ is a strong generating set of $H$ with respect to $B$. 
First note that since each $x \in X_i \cap H_{(\Delta_i)}$ is also in $X_{i+1}$ and a siftee of $x \in X_i \backslash H_{(\Delta_i)}$ by $H_{(\Delta_i)}$ is not in $H_{(\Delta_i)}$, we have 
$X_i \cap H_{(\Delta_i)} = X_{i+1} \cap H_{(\Delta_i)}$. 
Let $0 \leq s \leq m$ and consider $H^{[s+1]} = H_{(\beta_1, \beta_2,  \ldots, \beta_s)}$. 
Since $B$ is an orbit-ordered base, either $H^{[s+1]} \subseteq H_{(\Delta_i)}$ or $H_{(\Delta_i)} \subseteq H^{[s+1]}$. For both cases, we will show that $H^{[s+1]}= \langle X_{i+1} \cap H^{[s+1]} \rangle$.  \\
Suppose first that $H^{[s+1]} \subseteq H_{(\Delta_i)}$. Then
\[X_i \cap H^{[s+1]} =    X_i \cap H_{(\Delta_i)} \cap H^{[s+1]} = X_{i+1} \cap H_{(\Delta_i)} \cap H^{[s+1]} =  X_{i+1} \cap H^{[s+1]}. \] 
Since $X_i$ is a strong generating set, we have $H^{[s+1]}= \langle X_i \cap H^{[s+1]} \rangle = \langle X_{i+1} \cap H^{[s+1]} \rangle$. \\ 
Suppose instead that $H_{(\Delta_i)} \subseteq H^{[s+1]}$.
Certainly $\langle X_{i+1} \cap H^{[s+1]} \rangle \leq H^{[s+1]}$. 
To show the reverse containment, since $X_i$ is a strong generating set, it suffices to show that $X_i \cap  H^{[s+1]} \subseteq \langle X_{i+1} \cap H^{[s+1]} \rangle$. 
Let $x \in X_i \cap H^{[s+1]}$. 
If $x \in  H_{(\Delta_i)}$, then by the construction of $X_{i+1}$, we have $x \in X_{i+1}$, hence 
$x \in X_{i+1} \cap H^{[s+1]}$. 
Else suppose that $x \in (X_i \cap H^{[s+1]}) \backslash H_{(\Delta_i)}$. 
Then there exists a siftee $x'$ of $x$ by $H_{(\Delta_i)}$ such that $x' \in X_{i+1}$. 
By \Cref{props of siftee by stab}, there exists $g \in H_{(\Delta_i)}\leq H^{[s+1]}$ such that $x = x' g$.
Then $x'=xg^{-1} \in H^{[s+1]}$ and so $x' \in X_{i+1} \cap H^{[s+1]}$. 
\ifthenelse{\equal{\thesis}{1}}{
Observe that
\begin{eqnarray*}
H_{(\Delta_i)} &=&  \langle X_i \cap H_{(\Delta_i)} \rangle \text{ since $X_i$ is a strong generating set} \\
&=& \langle X_{i+1} \cap H_{(\Delta_i)} \rangle \text{ since $X_i \cap H_{(\Delta_i)} = X_{i+1} \cap H_{(\Delta_i)}$} \\
&\subseteq& \langle X_{i+1} \cap H^{[s+1]} \rangle \text{ since $H_{(\Delta_i)} \subseteq H^{[s+1]}$.}
\end{eqnarray*}
So $g \in \langle X_{i+1} \cap H^{[s+1]} \rangle$ and hence $x \in \langle X_{i+1} \cap H^{[s+1]} \rangle$. 
}{
Since \[ g \in H_{(\Delta_i)} =  \langle X_i \cap H_{(\Delta_i)} \rangle 
= \langle X_{i+1} \cap H_{(\Delta_i)} \rangle 
\subseteq \langle X_{i+1} \cap H^{[s+1]} \rangle,  \]}
we have $x \in \langle X_{i+1} \cap H^{[s+1]} \rangle$. 
\end{proof}

\begin{lemma} \label{iterative step in ddpd correct}
Let $1 \leq i \leq k-1$.
Let $X_i$ be an ${i}$-separable generating set of $H$ and let $\mathcal{P}_i = \langle C_1 \mid C_2 \mid \ldots \mid C_r \rangle$ be as in \Cref{defin: i partition}. 
Let $X_{i+1}$ and $Q$ be the output of $DDPD(i, B, X_i, \mathcal{P}_i)$ from \Cref{algm: DDPD}. 
Then 
\begin{enumerate}
    \item \label{iteratively computes ddfd} $\mathcal{P}_{i+1} = Q$. 
    \item \label{maintain separable} $X_{i+1}$ is an $(i+1)$-separable generating set of $H$.
\end{enumerate}
Hence \Cref{algm: DDPD} is correct. 
\end{lemma}

\begin{proof}
\Cref{iteratively computes ddfd}: 
We show that the set $S$ constructed from \Cref{algm: DDPD} is the same as the set in \Cref{finest ddfd i+1 - partn ver}, from which the result will follow. 
Observe that we add $C_j$ to $S$ in line 10 if and only if there exists $x \in X_i$ with non-trivial projection $\Proj_{\overline{i}}(x)$ such that
$\Proj_{\overline{i}}(x) \in \Proj_{C_j}(H) \times 1_{\overline{i} \backslash C_j}$ and
$\proj_{i+1}(x') \neq 1$, where $x'$ is a siftee of $x$ by $H_{(\Delta_i)}$. 
Since $X_i$ is $i$-separable, by \Cref{get theta from sgs,image of theta by siftee}, 
such an $x \in X_i$ exists if and only if $\theta_i(\Proj_{C_j}(H) \times 1_{\overline{i}\backslash C_j}) \neq N_{i+1}$. \ifthenelse{\equal{\thesis}{1}}{That is, $\Proj_{C_j}(H) \times 1_{\overline{i} \backslash C_j} \not \subseteq \mathrm{Ker}(\theta_i)$.}{}
The result then follows from \Cref{finest ddfd i+1 - partn ver}. \\
\Cref{maintain separable}: Let $y \in X_{i+1}$. We show that if $\Proj_{\overline{i+1}}(y) \neq 1$, then there exists a unique cell $C'$ of $\mathcal{P}_{i+1}$ such that $\Proj_{\overline{i+1}}(y)  \in \Proj_{C'}(H) \times 1_{\overline{i+1} \backslash C'}$.  \\
By the construction of $X_{i+1}$, either $y \in X_i \cap H_{(\Delta_i)}$ or there exists $x \in X_i \backslash H_{(\Delta_i)}$ such that $y$ is a siftee of $x$ by $H_{(\Delta_i)}$. 
Suppose first that $y \in X_i \cap H_{(\Delta_i)}$, so $\Proj_{\overline{i}}(y) = 1$. 
If $\Proj_{\overline{i+1}}(y) \neq 1$, then $\proj_{i+1}(y) \neq 1$ and so $\Proj_C(y) \neq 1$, where $C$ is defined as in line 16. 
Since $\overline{i+1} \backslash C = \overline{i} \backslash C \subseteq \overline{i}$, we have $\Proj_{\overline{i+1} \backslash C}(y) = 1$. So $C$ is the unique cell of $\mathcal{P}_{i+1}$ such that $\Proj_{\overline{i+1}}(y) = \Proj_{C}(y) \times 1_{\overline{i+1} \backslash C} \in \Proj_{C}(H) \times 1_{\overline{i+1} \backslash C}$. \\
Suppose now that $y$ is a siftee of some $x \in X_i \backslash H_{(\Delta_i)}$ by $H_{(\Delta_i)}$.
By \Cref{props of siftee by stab}, there exists $g \in H_{(\Delta_i)}$ such that $x = y g$. 
Then since $\Proj_{\overline{i}}(g)=1$, we have $\Proj_{\overline{i}}(x) = \Proj_{\overline{i}}(y)$. 
As $X_i$ is ${i}$-separable, there exists exactly one cell $C_j$ of $\mathcal{P}_i$ such that 
\ifthenelse{\equal{\thesis}{1}}{
\[ \Proj_{\overline{i}}(y) = \Proj_{\overline{i}}(x) \in 
\Proj_{C_j}(H) \times 1_{\overline{i} \backslash C_j}. 
\]}{$\Proj_{\overline{i}}(y) = \Proj_{\overline{i}}(x) \in 
\Proj_{C_j}(H) \times 1_{\overline{i} \backslash C_j}$.  \\}
If $C_j  \not \in S$, then $C_j$ is a cell of $\mathcal{P}_{i+1}$ and so $\Proj_{\overline{i+1}}(y) \in \Proj_{C_j}(H) \times 1_{\overline{i+1} \backslash C_j}$. 
Otherwise if $C_j \in S$, then 
$\Proj_{\overline{i+1}}(y) \in \Proj_{C_j \cup \{i+1\}}(H) \times 1_{\overline{i} \backslash C_j}$. 
In particular, since $C_j \cup \{i+1\} \subseteq C$, we have $\Proj_{\overline{i+1}}(y) \in \Proj_{C}(H) \times 1_{\overline{i+1} \backslash C}$. 
\end{proof}

\ifthenelse{\equal{\thesis}{1}}{

\begin{theorem}\label{thm: ddpd in poly time}
Given $H \leq S_n$ by the generating set $X$. Then the finest disjoint direct product decomposition of $H$ can be computed in time polynomial in $|X|n$.
\end{theorem}

\begin{proof}
Let $Q \coloneqq \langle C_1 \mid C_2 \mid \ldots \mid C_r \rangle$ be the partition computed by \Cref{algm: DDPD overall}. 
Then by \Cref{iteratively computes ddfd} of \Cref{iterative step in ddpd correct}, $Q = \mathcal{P}_k$, and so
$H = \Proj_{C_1}(H) \times \Proj_{C_2}(H) \times \ldots \times \Proj_{C_r}(H)$ is the finest disjoint direct product decomposition of $H$. \\
It remains to show that \Cref{algm: DDPD overall} runs in polynomial time. 
By \Cref{elementary poly time} and \Cref{schreier sims in poly time}, lines 1 and 3 run in polynomial time. 
By \Cref{concat orbs gives orbit ordered base}, we can get an orbit-ordered base in polynomial time. 
Since \Cref{algm: DDPD} is called at most $O(n)$ times, it remains to show that \Cref{algm: DDPD} runs in time polynomial in $|X|n$. \\
By \Cref{pointstab,membership test} of \Cref{subsections: consequences of Scherier Sims}, lines 5 and 7 can be done in polynomial time. 
Line 6 can be done by iterating through each cell $C_j$ of $\mathcal{P}_i$ and checking if $\Proj_{C_j}(x) \neq 1$. Since $\mathcal{P}_i$ has at most $k$ cells, this can be done in $O(k)$ time. 
Therefore the result follows. 
\end{proof}

}{
The following results are fundamental in permutation group algorithms. For more information, refer to, for example, \cite{seress, handbook}. 

\begin{lemma}\label{list of poly algm}
Given $K \leq S_n$ by a generating set $Y$, the following permutation group algorithms run in time polynomial in $n \cdot |Y|$:
\begin{enumerate}
    \item Computing orbits of $K$.
    \item Computing a strong generating set $X$ of $K$ with respect to a given base $B$.
    \item Computing pointwise stabilisers of $K$.
    \item Obtaining a siftee of $x \in S_n$ by $K$ using the sifting procedure. 
    \item Testing membership of $x \in S_n$ in $K$. 
\end{enumerate}
\end{lemma}


\begin{theorem}
Given $H \leq S_n$ by the generating set $X$, the finest disjoint direct product decomposition of $H$ can be computed in time polynomial in $n \cdot |X|$.
\end{theorem}

\begin{proof}
Let $Q \coloneqq \langle C_1 \mid C_2 \mid \ldots \mid C_r \rangle$ be the partition computed by \Cref{algm: DDPD overall}. 
Then by \Cref{iteratively computes ddfd} of \Cref{iterative step in ddpd correct}, $Q = \mathcal{P}_k$, and so
$H = \Proj_{C_1}(H) \times \Proj_{C_2}(H) \times \ldots \times \Proj_{C_r}(H)$ is the finest disjoint direct product decomposition of $H$. \\
It remains to show that \Cref{algm: DDPD overall} runs in polynomial time. 
By \Cref{list of poly algm}, lines 1 and 3 run in polynomial time. Furthermore $|X_1|$ is of polynomial size. 
By \Cref{concat orbs gives orbit ordered base}, we can get an orbit-ordered base in polynomial time. 
Since \Cref{algm: DDPD} is called at most $O(n)$ times, it suffices to show that \Cref{algm: DDPD} runs in time polynomial in $n \cdot |X|$. \\
By \Cref{list of poly algm}, lines 5 and 7 can be done in polynomial time. 
Line 6 can be done by iterating through each cell $C_j$ of $\mathcal{P}_i$ and checking if $\Proj_{C_j}(x) \neq 1$. Since $\mathcal{P}_i$ has at most $k$ cells, this can be done in $O(k)$ time. 
Lastly since $|X_{i+1}| = |X_i|$, the result follows. 
\end{proof}

}

\section{Experiments}
\label{section: experiments}

In this section, we will investigate the practical performance of our algorithm for finding the finest disjoint direct product decomposition of a permutation group. As well as illustrating the performance of our algorithm, we will show how it can be used to improve the performance of a range of GAP implementations to solve important group-theoretic problems. Our algorithm is implemented in GAP 4.11 \cite{GAP4}.

\ignore{
There are two main reasons we might try to find a direct product decomposition of a group \(G\). We may be interested in better understanding the structure of \(G\), or we may hope the knowledge of the direct product decomposition will help speed up other calculations. When we are using the direct product decomposition to speed up other calculations, it is important that the cost of the decomposition is smaller than the time saved.

We experiment with two families of groups:

\begin{enumerate}
    \item The isomorphism groups of graphs used as subgraph isomorphism benchmarks. These graphs come from various sources.
    \item Randomly generated groups.
\end{enumerate}

\subsection{Graph Isomorphism}

We will consider 3 classes of graphs from \cite{}. For each class, we show the time taken to generate the generators for the isomorphism group of the graph (which also calculates the size of the group). The then we show the time taken to run our algorithm split into two pieces -- the time taken to build the required base and strong generating set, and the time taken to run the rest of the algorithm.

We generate the automorphisms using the Bliss system, which is part of the Digraph package of GAP. This system seems to take a surprisingly long time on some graphs, as shown in our results. It returns the exact size of the group, which is created, which greatly improves the performance when then finding a base and strong generating set.

We can see the time taken to run our algorithm is, on average, smaller than the time taken to either generate the isomorphisms or calculate a base and strong generating set. This shows that assuming a base and strong generating set is required, it is reasonable to then always run our algorithm. In this case, we were able to break the group into many smaller pieces.

\begin{tabular}{cccccc}
    \hline
    Class & Time & Time & Time  & Size & Size \\
    &(Graph Iso)&(Stab Chain) & (Decomp) & (Orig) & (Sum Decomp) \\
    \hline
    images-CVIU11 & 65.5 & 73.8 & 4.1 & $1.2*10^9$ & 31 \\
    biochemicalReactions & 3.6 & 226.4 & 1.57 & $2.1*10^{37}$ & $5.8*10^{14}$\\
    meshes-CVIU11 & 3381.7 & 1,034,318 & 147.84 & $1.0*10^{1130}$ & $4.3 * 10^{950}$ \\
    \hline
    
\end{tabular}
}

We will test our algorithm on randomly created groups. The creator we use in this \ifthenelse{\equal{\thesis}{1}}{chapter}{paper} is very straightforward and does not claim to produce all groups with equal probability. The creator takes three parameters, a transitive permutation group \(G\) and two integer constants \(r\) and \(s\). The algorithm will produce a permutation group which is the direct product of \(r\) groups. Each of these direct factors will be a subdirect product of \(G^s\) that is d.d.p.\ indecomposable. The algorithm runs in two stages.

The first stage is implemented by a function \textsc{MakeSubdirect(\(G,s\))} which produces a random subdirect product of \(s\) copies of \(G\) that is d.d.p.\ indecomposable. \textsc{Makesubdirect(\(G,s\))} works by taking a random integer \(i \in \{2,3,\dots,s\}\)  and then taking the group \(H\) generated by \(i\) random elements of \(G^s\). If \(H\) is d.d.p.\ indecomposable, and its projection onto each of the \(s\) copies of \(G\) is surjective, then it is returned, else this procedure repeats.

The second stage simply runs \textsc{MakeSubdirect(\(G,s\))} \(r\) times, and takes the direct product of these \(r\) groups. Finally, we conjugate this group by a random permutation on the set of points moved by this group, so the decomposition does not follow the natural ordering of the integers.

Our algorithmic section is split into two sections. Firstly we compare our algorithm against the algorithm of Donaldson and Miller \cite{donaldson_miller_2009}. Secondly, to demonstrate that our algorithm has immediate practical value, we show how a few functions in GAP can easily be sped up by the knowledge of a disjoint direct product decomposition.

\subsection{Comparison to Donaldson and Miller}
 
We first compare to Donaldson and Miller's algorithm. Donaldson and Miller present two algorithms, an incomplete algorithm and a complete algorithm. We will not compare against their incomplete algorithm, as it is extremely fast but requires separable strong generating sets.  Donaldson and Miller were unable to find a graph where the generating set of the automorphism group produced by nauty \ifthenelse{\equal{\thesis}{1}}{\cite{mckay}}{\cite{McKay2014}} is not separable. We were able to find graphs where nauty does not produce a separable generating set\footnote{Some examples and how one could look for more examples can be found in the supplemental files.}. Also, the most advanced graph automorphism finders, such as Traces \ifthenelse{\equal{\thesis}{1}}{\cite{mckay}}{\cite{McKay2014}}, perform random dives which produce random automorphisms. This means generating sets produced by Traces will not, in general, be separable.

When implementing Donaldson and Miller's complete algorithm, we were forced to make some implementation choices. Most significant is the ordering in which the algorithm tries to partition the orbits (the first line of Algorithm 5 in \cite{donaldson_miller_2009}). We implemented this by trying the partitions in order of increasing size of the smaller part of the partition. This has the advantage that when the finest disjoint direct product decomposition has a factor with few orbits the algorithm will run relatively quickly -- it does not affect the algorithm's worst case complexity.


We gather our results in \Cref{fig:AlgCompare}. 
Each row of \Cref{fig:AlgCompare} gives the result for 10 instances of $H$, created as above. Each instance of $H$ has $rs$ orbits with the finest disjoint direct product decomposition consisting of $r$ direct factors with $s$ orbits each, and each transitive constituent of $H$ is permutation isomorphic to $G$.
For each $G, r$ and $s$, we report the median time (``Median'') and the number of completed instances (``\#'') of the 10 instances of $H$, using both Algorithm 5 of \cite{donaldson_miller_2009} 
(``Donaldson'') and \Cref{algm: DDPD overall} (``Our Algorithm'').
The results are consistent with the theory. Since our algorithm is polynomial, it scales much better than the algorithm in \cite{donaldson_miller_2009}. 

\begin{figure}[ht]
\centering
\ifthenelse{\equal{\thesis}{1}}{\renewcommand{\arraystretch}{1.25}}{}
\begin{tabular}{ccccccc}
\hline
G&r&s&\multicolumn{2}{c}{Donaldson}&\multicolumn{2}{c}{Our Algorithm}\\
&&&Median&\#&Median&\#\\
\hline
$D_8$ & 4 & 4 & 0.04 & 10 &  0.00 & 10 \\
$D_8$ & 6 & 4 & 0.16 & 10 &  0.01 & 10 \\
$D_8$ & 8 & 4 & 0.54 & 10 &  0.01 & 10 \\
$D_8$ & 10 & 4 & 2.18 & 10 &  0.02 & 10 \\
\hline
$A_4$ & 4 & 4 & 2.47 & 10 &  0.01 & 10 \\
$A_4$ & 6 & 4 & 12.47 & 10 &  0.03 & 10 \\
$A_4$ & 8 & 4 & 158.09 & 10 &  0.04 & 10 \\
$A_4$ & 10 & 4 & 490.54 & 5 &  0.05 & 10 \\
\hline
$S_4$ & 4 & 4 & 3.37 & 10 &  0.02 & 10 \\
$S_4$ & 6 & 4 & 15.99 & 10 &  0.04 & 10 \\
$S_4$ & 8 & 4 & 393.86 & 8 &  0.06 & 10 \\
$S_4$ & 10 & 4 & N/A & 0 &  0.07 & 10 \\
\hline
\end{tabular}
\caption{Comparison of DDPD algorithms} 
\label{fig:AlgCompare}
\end{figure}

\subsection{Application to GAP}

In this section, we will see how some functions in GAP can easily be sped up by the knowledge of a disjoint direct product decomposition. We do not claim that this is an exhaustive list, but we intend to demonstrate how, for a selection of common problems, calculating a disjoint direct product decomposition can significantly improve performance. We experimented with three GAP functions: 

\begin{description}
    \item[\textsc{DerivedSubgroup}] computes the derived subgroup of a given group. The derived subgroup of a group $H$ is the direct product of the derived subgroup of each disjoint direct factor of $H$.
    \item[\textsc{NrConjugacyClasses}] gives the number of conjugacy classes of a given group. The number of conjugacy classes of a group $H$ is the product of the number of conjugacy classes of each disjoint direct factor of $H$.
    \item[\textsc{CompositionSeries}] computes a composition series of a given group. A composition series of a group $H$ can be constructed from composition series of the direct factors of $H$. 
\end{description}

In our experiments, we run each row of our tables ten times. Each run is given a limit of 10 minutes and 4GB of memory. We give the median time in seconds (or N/A when less than 6 instances finished successfully). For the inner group \(G\) we consider the alternating group ($A_n$), symmetric group ($S_n$), dihedral groups ($D_{2n}$) of varying degree \(n\), and also 
$\textsc{TransitiveGroup}(16,712)$ ($Trans(16,712)$) and $\textsc{TransitiveGroup}(16,713)$ ($Trans(16,713)$) from the Transitive Groups Library \cite{TransGrpGAPlib}. 

Each row of the tables in \Cref{fig:DerivedSubgroup,fig:NrConjugacyClasses,fig:Composeries}  gives the results for 10 random groups $H$, each with $rs$ orbits, 
where the projection of $H$ onto each orbit is permutation isomorphic to $G$, 
and $H$ has the finest disjoint direct product decomposition consisting of $r$ direct factors with $s$ orbits each. 
The columns ``Full Group'' and ``Decomposed Group'' refer to the computation with the original group $H$ and the computation with the disjoint direct product factors of $H$ respectively. The column ``Decomposition'' refers to the computation of the finest disjoint direct product decomposition of $H$. 
For each of these columns, we report the median time (in seconds) required to compute the specified problems of the 10 instances under the subcolumns ``Median'', 
and the number of instances completed within the time and memory limits under the subcolumns ``\#''. 

The time taken to find the finest disjoint direct product decomposition and solve \textsc{DerivedSubgroup} (\Cref{fig:DerivedSubgroup}) and \textsc{NrConjugacyClasses}  (\Cref{fig:NrConjugacyClasses}) on the decomposed group are always faster than solving the problem on the original full group. 
In the case of \textsc{DerivedSubgroup} (\Cref{fig:DerivedSubgroup}), we speed up performance by up to a factor of 10.  In the case of \textsc{NrConjugacyClasses}  (\Cref{fig:NrConjugacyClasses}) 
and \textsc{CompositionSeries} (\Cref{fig:Composeries}), we are able to solve problems that previously ran out of memory or time, in under a second. Almost the entire time taken by Decomposition is building an initial stabiliser chain with respect to an orbit-ordered base. For example, in the largest problem in \Cref{fig:DerivedSubgroup}, building the stabiliser chain took 22.5 of the 23.7 seconds taken in total. The time taken to build a stabiliser chain with respect to an orbit-ordered base is, on average, no longer than the time taken to build a stabiliser chain using GAP's default strategy for base ordering.

\begin{figure}[hpt!]
\centering
\ifthenelse{\equal{\thesis}{1}}{\renewcommand{\arraystretch}{1.25}}{}
\begin{tabular}{ccccccccc}
\hline
G&r&s&\multicolumn{2}{c}{Full Group}&\multicolumn{2}{c}{Decomposed Group}&\multicolumn{2}{c}{Decomposition}\\
&&&Median&\#&Median&\#&Median&\#\\
\hline
$A_4$ & 12 & 4 & 2.34 & 10 & 0.01 & 10 & 0.28 & 10 \\
$A_4$ & 16 & 4 & 7.53 & 10 & 0.01 & 10 & 0.66 & 10 \\
$A_4$ & 20 & 4 & 19.76 & 10 & 0.01 & 10 & 1.64 & 10 \\
\hline
$D_8$ & 12 & 4 & 0.33 & 10 & 0.00 & 10 & 0.20 & 10 \\
$D_8$ & 16 & 4 & 1.07 & 10 & 0.00 & 10 & 0.55 & 10 \\
$D_8$ & 20 & 4 & 2.53 & 10 & 0.00 & 10 & 1.20 & 10 \\
\hline
$S_4$ & 12 & 4 & 7.36 & 10 & 0.01 & 10 & 0.89 & 10 \\
$S_4$ & 16 & 4 & 27.89 & 10 & 0.02 & 10 & 2.34 & 10 \\
$S_4$ & 20 & 4 & 65.85 & 10 & 0.02 & 10 & 5.37 & 10 \\
\hline
$D_{32}$ & 12 & 4 & 6.58 & 10 & 0.01 & 10 & 1.02 & 10 \\
$D_{32}$ & 16 & 4 & 18.47 & 10 & 0.01 & 10 & 2.86 & 10 \\
$D_{32}$ & 20 & 4 & 44.72 & 10 & 0.02 & 10 & 6.47 & 10 \\
\hline
$Trans(16,712)$ & 12 & 4 & 17.34 & 10 & 0.03 & 10 & 4.46 & 10 \\
$Trans(16,712)$ & 16 & 4 & 56.20 & 10 & 0.04 & 10 & 11.61 & 10 \\
$Trans(16,712)$ & 20 & 4 & 144.37 & 10 & 0.05 & 10 & 27.43 & 10 \\
\hline
$Trans(16,713)$ & 12 & 4 & 13.53 & 10 & 0.03 & 10 & 3.69 & 10 \\
$Trans(16,713)$ & 16 & 4 & 38.17 & 10 & 0.05 & 10 & 10.80 & 10 \\
$Trans(16,713)$ & 20 & 4 & 90.89 & 10 & 0.06 & 10 & 23.72 & 10 \\
\hline
\end{tabular}
\caption{Performance of DDPD for \textsc{DerivedSubgroup}}
\label{fig:DerivedSubgroup}
\end{figure}

\begin{figure}[hpt!]
\centering
\ifthenelse{\equal{\thesis}{1}}{\renewcommand{\arraystretch}{1.25}}{}
\begin{tabular}{ccccccccc}
\hline
G&r&s&\multicolumn{2}{c}{Full Group}&\multicolumn{2}{c}{Decomposed Group}&\multicolumn{2}{c}{Decomposition}\\
&&&Median&\#&Median&\#&Median&\#\\
\hline
$A_3$ & 2 & 4 & 0.00 & 10 & 0.00 & 10 & 0.00 & 10 \\
$A_3$ & 4 & 4 & 0.08 & 10 & 0.00 & 10 & 0.00 & 10 \\
$A_3$ & 6 & 4 & 21.50 & 6 & 0.00 & 10 & 0.00 & 10 \\
\hline
$D_8$ & 2 & 4 & 0.23 & 10 & 0.00 & 10 & 0.00 & 10 \\
$D_8$ & 4 & 4 & N/A & 2 & 0.00 & 10 & 0.00 & 10 \\
$D_8$ & 6 & 4 & N/A & 0 & 0.01 & 10 & 0.00 & 10 \\
\hline
$S_4$ & 2 & 4 & 50.52 & 10 & 0.28 & 10 & 0.00 & 10 \\
$S_4$ & 4 & 4 & N/A & 0 & 0.58 & 10 & 0.00 & 10 \\
$S_4$ & 6 & 4 & N/A & 0 & 0.96 & 10 & 0.02 & 10 \\
\hline
\end{tabular}
\caption{Performance of DDPD for \textsc{NrConjugacyClasses}}
\label{fig:NrConjugacyClasses}
\end{figure}

\begin{figure}[hpt!]
\centering
\ifthenelse{\equal{\thesis}{1}}{\renewcommand{\arraystretch}{1.25}}{}
\begin{tabular}{ccccccccc}
\hline
G&r&s&\multicolumn{2}{c}{Full Group}&\multicolumn{2}{c}{Decomposed Group}&\multicolumn{2}{c}{Decomposition}\\
&&&Median&\#&Median&\#&Median&\#\\
\hline
$A_4$ & 4 & 3 & 0.00 & 10 & 0.00 & 10 & 0.00 & 10 \\
$A_4$ & 12 & 3 & 0.06 & 10 & 0.01 & 10 & 0.11 & 10 \\
$A_4$ & 20 & 3 & 0.22 & 10 & 0.02 & 10 & 0.68 & 10 \\
\hline
$S_4$ & 4 & 3 & 0.01 & 10 & 0.00 & 10 & 0.00 & 10 \\
$S_4$ & 12 & 3 & 0.15 & 10 & 0.02 & 10 & 0.32 & 10 \\
$S_4$ & 20 & 3 & 0.49 & 10 & 0.03 & 10 & 1.70 & 10 \\
\hline
$D_8$ & 4 & 3 & 0.00 & 10 & 0.00 & 10 & 0.00 & 10 \\
$D_8$ & 12 & 3 & 0.03 & 10 & 0.01 & 10 & 0.07 & 10 \\
$D_8$ & 20 & 3 & 0.10 & 10 & 0.02 & 10 & 0.40 & 10 \\
\hline
$D_{32}$ & 4 & 3 & 0.01 & 10 & 0.00 & 10 & 0.00 & 10 \\
$D_{32}$ & 12 & 3 & 0.11 & 10 & 0.02 & 10 & 0.27 & 10 \\
$D_{32}$ & 20 & 3 & 0.45 & 10 & 0.04 & 10 & 1.87 & 10 \\
\hline
$Trans(16,712)$ & 4 & 3 & 0.02 & 10 & 0.01 & 10 & 0.03 & 10 \\
$Trans(16,712)$ & 12 & 3 & 0.36 & 10 & 0.03 & 10 & 1.25 & 10 \\
$Trans(16,712)$ & 20 & 3 & 1.52 & 10 & 0.06 & 10 & 8.72 & 10 \\
\hline
$Trans(16,713)$ & 4 & 3 & 0.64 & 10 & 0.02 & 10 & 0.03 & 10 \\
$Trans(16,713)$ & 12 & 3 & 80.49 & 10 & 0.07 & 10 & 1.32 & 10 \\
$Trans(16,713)$ & 20 & 3 & N/A & 0 & 0.12 & 10 & 8.57 & 10 \\
\hline
\end{tabular}
\caption{Performance of DDPD for \textsc{CompositionSeries}}
\label{fig:Composeries}
\end{figure}

\section{Conclusion and future work}
\label{section: conclusion}

In this paper, we have shown that the finest disjoint direct product decomposition of a given group can be computed efficiently and can be used to speed up various permutation group problems.
Moreover, as demonstrated in \cite{donaldson_miller_2009, graylandThesis}, the disjoint direct product decomposition of a group has applications beyond computational group theory. 

While we show the disjoint direct product decomposition can be extremely useful, we are not suggesting it to be employed as an initial subprocedure of solving all the problems we use in our experiments. 
This is because adding this subprocedure will impose additional computation time on all calls of the problem that could be a waste of time if the group is d.d.p.\ indecomposable.
Using efficient heuristics to only add this subprocedure for some groups still has the same problem with the added cost and raises an issue of determining the heuristics.
Therefore, we propose that the computation of the finest disjoint direct product decomposition be available in GAP as a function, and leave it up to the user to decide if the decomposition would help the problem on hand.

An obvious piece of future work is to determine other problems, group theoretic or otherwise, that can benefit from the knowledge of its disjoint direct product decomposition. For example, the first author has found applications in her work with computing normalisers and the second author in his work with graph isomorphisms. 
We believe that the disjoint direct product decomposition has more potential in groups arising from real world problems, as these are more likely to be highly intransitive. 

\section*{Acknowledgement}
\label{section: acknowledgement}

The authors would like to thank the reviewers, Colva Roney-Dougal and Wilf
Wilson for their useful comments. The second author is supported by a Royal Society University Research Fellowship.


\bibliographystyle{alpha}
\bibliography{bib}

\end{document}